\documentclass[a4paper, oneside,11 pt]{article}

\usepackage[latin1]{inputenc}
\usepackage[T1]{fontenc}
\usepackage[english]{babel}

\usepackage[all]{xy}
\usepackage{verbatim}
\usepackage{amsmath}
\usepackage{amsthm}
\usepackage{graphicx}
\usepackage{amssymb}
\usepackage{epstopdf}
\usepackage{url}
\usepackage{amsfonts}
\usepackage{todonotes}
\usepackage{fullpage}

\newcommand{\pt}{\partial}

\newcommand{\vp}{\varphi}

  \newcommand{\M}{{\mathcal M}}
\newcommand{\br}{\mathbb{R}}
\newcommand{\bz}{\mathbb{Z}}
\newcommand{\bn}{\mathbb{N}}
\newcommand{\bt}{\mathbb{T}}
\renewcommand{\L}{\mathbb{L}}

\renewcommand{\a}{\alpha}

\newcommand{\e}{\varepsilon}

\renewcommand{\(}{\left(}
\renewcommand{\)}{\right)}

\newcommand{\na}{\nabla}

\newcommand{\curl}{\operatorname{curl}}
\renewcommand{\div}{\operatorname{div}}

\newtheorem{thm}{Theorem}
\newtheorem{lem}[thm]{Lemma}

\newtheorem{prop}[thm]{Proposition}
\newtheorem{defi}[thm]{Definition}
\newtheorem{remark}[thm]{Remark}

\def\be{\begin{equation}}
\def\ee{\end{equation}}
\def\bea{\begin{eqnarray}}
\def\eea{\end{eqnarray}}

\numberwithin{thm}{section}
\numberwithin{equation}{section}

\author{Daniel Han-Kwan\footnote{CNRS $\&$ \'Ecole Polytechnique, Centre de Math\'ematiques Laurent Schwartz, UMR 7640. Email: daniel.han-kwan@math.polytechnique.fr}, \, Mikaela Iacobelli\footnote{Universit\`a ``La Sapienza'', Dipartimento di Matematica ``Guido Castelnuovo'', Roma.
Email: iacobelli@mat.uniroma1.it} \footnote{
\'Ecole Polytechnique, Centre de Math\'ematiques Laurent Schwartz, UMR 7640. 
}
}

\title{Quasineutral limit for Vlasov-Poisson via Wasserstein stability estimates in higher dimension}

\begin{document}

\maketitle

\begin{abstract}
This work is concerned with the quasineutral limit of the Vlasov-Poisson system in two and three dimensions.
We justify the formal limit for very small but rough perturbations of analytic initial data, generalizing the results of \cite{HI} to higher dimension.
\end{abstract}

\section{Introduction}
In a non relativistic setting the dynamics of electrons in a plasma with heavy ions uniformly distributed in space is described by the Vlasov-Poisson system. Throughout this paper, we will focus on the $2$ and $3$ dimensional periodic (in space) case.
We introduce the distribution function of the electrons $f(t, x, v)$, for $t \in \br^+$, $(x,v)\in \bt^d\times \br^d$ where $ \bt^d$ is the $d$-dimensional torus and $d=2,3.$ As usual, $f(t, x, v)dxdv$ can be interpreted as the probability of finding particles with position and velocity close to the point $(x,v)$ in the phase space at time $t$. We also define the electric potential $U(t, x)$ and the associated electric field $E(t,x)$. 

We introduce the positive parameter $\e$ defined as the ratio of the \emph{Debye length} of the plasma to the size of the domain. The Debye length can be interpreted as the typical length below which charge separation occurs; it plays an important role in plasma physics. For a more detailed discussion on this subject we refer to the introduction of \cite{HK}. 
Adding a subscript in order to emphasize on the dependance on $\e$, we end up with the rescaled Vlasov-Poisson system:
\be
\label{vp}
 \left\{ \begin{array}{ccc}\pt_t f_\e+v\cdot \na_x f_\e+ E_\e\cdot \na_v f_\e=0,  \\
E_\e=-\na_x U_\e, \\
-\e^2 \Delta_x U_\e=\int_{\br^d} f_\e\, dv - \int_{\bt^d \times \br^d}  f_\e\, dv \, dx ,\\
f_\e\vert_{t=0}=f_{0,\e}\ge0,\ \  \int_{\bt^d \times \br^d} f_{0,\e}\,dx\,dv=1,
\end{array} \right.
\ee
and the energy of this system is
\be
\label{energy}
\mathcal{E}(f_\e(t)) := \frac{1}{2} \int_{\bt^d \times \br^d} f_\e |v|^2 \, dv dx + \frac{\e^2}{2} \int_{\bt^d} |\nabla_x U_\e|^2 \, dx.
\ee

In this paper we study the behavior of solutions to the system \eqref{vp} as $\e$ goes to $0$. We will refer to this limit as the \emph{quasineutral limit}. Let us notice that the Debye length is almost always very small compared to the typical observation length, so the quasineutral limit is relevant from the physical point of view and widely used in plasma physics.

Let us observe that, if $f_\e \to f$ and $U_\e \to U$ in some sense as $\e \to 0$, the formal limit of our system is 
\be
\label{formal}
 \left\{ \begin{array}{ccc}\pt_t f +v\cdot \na_x f+ E\cdot \na_v f=0,  \\
E= -\nabla_x U, \\
\int_{\br^d} f\, dv = 1,\\
f\vert_{t=0}=f_{0}\ge0,\ \  \int_{\bt^d \times \br^d} f_0\,dx\,dv=1,
\end{array} \right.
\ee
and the total energy of the system reduces to the kinetic part of \eqref{energy}
$$
\mathcal{E}(f(t)) := \frac{1}{2} \int_{\bt^d \times \br^d} f |v|^2 \, dv dx.
$$
In this system, the force $E$ is a Lagrange multiplier, or a pressure, associated to the constraint $\int_{\br^d} f\, dv = 1$.

The justification of the quasineutral limit from the rescaled Vlasov-Poisson system \eqref{vp} to \eqref{formal} is subtle and has a long history. Up to now, this limit is known to be true only in few cases and we refer to \cite{Br89,Gr95, Gr96, Br00,Mas, HKH, HI} for a deeper understanding of this problem.

One of the first mathematical works on the quasineutral limit of the Vlasov-Poisson system was performed by Grenier  in \cite{Gr96}. He introduces an interpretation of the plasma as a superposition of a -possibly uncountable- collection of fluids and he shows that the quasineutral limit holds when the sequence of initial data $f_{0,\e}$ enjoys uniform analytic regularity with respect to the space variable. This convergence result has been improved by Brenier \cite{Br00}, who gives a rigorous justification of  the quasineutral limit in the so called ``cold electron'' case, i.e. when the initial distribution $f_{0,\e}$ converges to a monokinetic profile 
$$
f_0(x,v) = \rho_0(x) \delta_{v= v_0(x)}
$$ 
where $\delta_v$ denotes the Dirac measure in velocity. For further insight on this direction see also \cite{Br00,Mas,GSR}. 

A different approach, more focused on the question of stability, or eventually instability, around homogeneous equilibria in the quasineutral limit is developed in \cite{HKH}. They show that the limit is true for homogeneous profiles that satisfy some monotonicity condition, together with a symmetry condition, i.e. when the initial distribution $f_{0,\e}$ converges to an homogeneous initial condition $\mu(v)$ which is symmetric with respect to some $\overline{v} \in \br$ and which is first increasing then decreasing. 

In our previous work \cite{HI}, we consider the quasineutral limit of the one-dimensional Vlasov-Poisson
equation for ions with massless thermalized electrons (considering that electrons move very fast and quasi-instantaneously reach their local thermodynamic equilibrium), and we prove that the limit holds for very small but rough perturbations of analytic data. In this context, small means small in the Wasserstein distance $W_1$, which implies that highly oscillatory perturbations are for instance allowed. Our aim is to show that an analogue of this result holds in higher dimension.
\medskip

Let us introduce the notions of $p$-Wasserstein distance and weak convergence in the $p$-Wasserstein space $\mathcal{P}_p(\mathcal{M})$ (see for instance \cite{Vil03}).
\begin{defi}[Wasserstein distance]
Let $(\mathcal{M}, d)$ be a Polish space and let us denote with $\mathcal{P}_p(\mathcal{M})$ the collection of all probability measures $\mu$ on $\mathcal{M}$ with finite $p$ moment: for some $x_0 \in \mathcal{M}$,
$$
\int_{\mathcal{M}} d(x, x_{0})^{p} \, \mathrm{d} \mu (x) < +\infty.
$$
Then the $p$-Wasserstein distance between two probability measures $\mu$ and $\nu$ in $\mathcal{P}_p(\mathcal{M})$ is defined as
$$
W_{p} (\mu, \nu):=\left( \inf_{\gamma \in \Gamma (\mu, \nu)} \int_{\mathcal{M} \times \mathcal{M}} d(x, y)^{p} \, \mathrm{d} \gamma (x, y) \right)^{1/p},
$$
where $\Gamma (\mu, \nu)$ denotes the collection of all measures on $\mathcal{M}\times \mathcal{M}$ with marginals $\mu$ and $\nu$ on the first and second factors respectively. 
\end{defi}

\begin{defi}[Weak convergence in $\mathcal{P}_p(\mathcal{M})$] Let $(\mathcal{M}, d)$ be a Polish space, and $p\in [1, \infty).$ Let $(\mu_k)_{k\in \bn}$ be a sequence of probability measures in $\mathcal{P}_p(\mathcal{M})$ and let $\mu$ be another measure in $P(\mathcal{M})$. Then $\mu_k$ converges weakly in $\mathcal{P}_p(\mathcal{M})$ to $\mu$ if any one of the following equivalent properties is satisfied for some (and then any) $x_0 \in \mathcal{M}$:
\begin{enumerate}
\item $\mu_k\rightharpoonup\mu$ and $\int d(x, x_0)^p d\,\mu_k(x)\to\int d(x, x_0)^p d\,\mu(x);$
\item $\mu_k\rightharpoonup\mu$ and $\underset{k\to\infty}{\limsup}\int d(x, x_0)^p d\,\mu_k(x)\le\int d(x, x_0)^p d\,\mu(x);$
\item $\mu_k\rightharpoonup\mu$ and $\underset{R\to\infty}{\lim} \underset{k\to\infty}{\limsup}\int_{d(x,x_0)\ge R}d(x, x_0)^p d\,\mu_k(x)=0;$
\item For all continuous functions $\vp$ with $|\vp(x)|\le C(1+d(x,x_0)^p),$ $C\in \br,$ one has
$$
\int \vp(x)d\,\mu_k(x)\to\int \vp(x)d\,\mu(x).
$$
\end{enumerate}

\end{defi}
In our case the \emph{Wasserstein space} $\mathcal{P}_2(\mathcal{M})$ is the space of probability measures which have a finite moment of order $2$ and it will always be equipped with the quadratic Wasserstein distance $W_2$.
\begin{remark}
\begin{itemize}
\item By H\"{o}lder's inequality we have that
$$
p \le q \Rightarrow W_p \le W_q.
$$
In particular, the Wasserstein distance $W_1$, is the weakest of all and results in $W_2$ distance are usually stronger than results in $W_1$ distance.
\item Let $(\mathcal{M}, d)$ be a Polish space, and $p\in [1, \infty);$ then the Wasserstein distance $W_p$ metrizes the weak convergence in $\mathcal{P}_p(\mathcal{M})$. In other words, if $(\mu_k)_{k\in \bn}$ is a sequence of measures in $\mathcal{P}_p(\mathcal{M})$ and $\mu$ is another measure in $P(\mathcal{M})$, then $\mu_k$ converges weakly in $\mathcal{P}_p(\mathcal{M})$ to $\mu$ if and only if
$$
W_p(\mu_k, \mu) \to 0 \qquad as\ \  k\to \infty.
$$
\end{itemize} 
\end{remark}
In order to state our main result, let us introduce the fluid point of view and the convergence result for uniformly analytic initial data  introduced by Grenier in \cite{Gr96}.

We first recall the definition of spaces of analytic functions we use below.
\begin{defi}
\label{def:norm}
Given $\delta>0$ and a function $g:\bt \to \mathbb R$, we define
$$
\| g \|_{B_\delta} := \sum_{k \in \bz} | \widehat{g}(k) | \delta^{|k|},
$$
where $\widehat{g}(k)$ is the $k$-th Fourier coefficient of $g$. We define $B_\delta$ as the space of functions $g$ such that $\| g \|_{B_\delta}<+\infty$.
\end{defi}


We assume that, for all $\e \in (0,1)$, $g_{0,\e}(x,v)$ is a \emph{continuous} function; following Grenier \cite{Gr96}, we write each initial condition as a ``superposition of Dirac masses in velocity'':
$$
g_{0, \e}(x,v) = \int_\M \rho_{0,\e}^\theta(x) \delta_{v= v_{0,\e}^\theta(x)} \, d\mu(\theta)
$$
with $\M:= \br^d$, $d\mu(\theta) = c_d \frac{d\theta}{1+|\theta|^{d+1}}$, 
where $c_d$ is a normalizing constant (depending only on the dimension $d$),
$$\rho_{0, \e}^\theta= \frac{1}{c_d} (1+ |\theta|^{d+1}) g_{0,\e}(x,\theta), \quad  v_{0,\e}^\theta = \theta.$$
This leads to the study of the behavior as $\e\to0$ for solutions to the\emph{ multi-fluid pressureless Euler-Poisson system} 
\begin{equation}
\label{fluid}
 \left\{ \begin{array}{ccc}\pt_t \rho_\e^\theta+ \nabla_x  \cdot (\rho_\e^\theta v_\e^\theta)=0,  \\
\pt_t v_\e^\theta + v_\e^\theta \cdot \nabla_x v_\e^\theta = E_\e, \\
E_\e=- \nabla_x U_\e, \\
-\e^2 \Delta_x U_\e=\int_\M \rho_\e^\theta \, d\mu(\theta) -1,\\
\rho_\e^\theta\vert_{t=0}= \rho_{0,\e}^\theta, v_\e^\theta\vert_{t=0}= v_{0,\e}^\theta.
\end{array} \right.
\end{equation}
One then checks that defining
$$g_\e(t,x,v) = \int_\M \rho_\e^\theta(t,x) \delta_{v= v_\e^\theta(t,x)} \, d\mu(\theta)$$
provides a weak solution to \eqref{vp}.

The formal limit system, which is associated to the kinetic incompressible Euler system \eqref{formal}, is the following \emph{multi fluid incompressible Euler system}:
\begin{equation}
\label{limit-fluid}
\left\{ \begin{array}{ccc}\pt_t \rho^\theta+  \nabla_x  \cdot  (\rho^\theta v^\theta)=0,  \\
\pt_t v^\theta + v^\theta \cdot   \nabla_x v^\theta = E, \\
\operatorname{curl} E= 0, \, \int_{\bt^d} E \, dx =0, \\
\int_\M \rho^\theta \, d\mu(\theta) =1,\\
\rho^\theta\vert_{t=0}= \rho_{0}^\theta, v^\theta\vert_{t=0}= v_{0}^\theta,
\end{array} \right.
\end{equation}
where the $\rho_0^\theta$ are defined as the limits of $\rho_{0,\e}^\theta$ (which are thus supposed to exist)  and $v_0^\theta=\theta$. 

As before, one checks that defining
$$g(t,x,v) = \int_\M \rho^\theta(t,x) \delta_{v= v^\theta(t,x)} \, d\mu(\theta)$$
gives a weak solution to the kinetic Euler incompressible system \eqref{formal}.


We are now in position to state
the results of \cite[Theorems 1.1.2, 1.1.3 and Remark 1 p. 369]{Gr96}. 

\begin{prop}
\label{grenier}
Assume that there exist $\delta_0,C,\eta>0$, with $\eta$ small enough,  such that
$$
\sup_{\e\in (0,1)}\sup_{v \in \br} (1+v^2) \| g_{0,\e} (\cdot,v)\|_{B_{\delta_0}}  \leq C,
$$
and that 
$$
\sup_{\e\in (0,1)} \left\|  \int_\br g_{0,\e}(\cdot,v) \, dv  -1 \right\|_{B_{\delta_0}} < \eta.
$$
%
Denote for all $\theta \in \br$,
$$\rho_{0, \e}^\theta= \pi (1+ \theta^2) g_{0,\e}(x,\theta), \quad  v_{0,\e}^\theta = v^\theta= \theta.$$
Assume that for all $\theta \in \br$, $\rho_{0,\e}^\theta$ has a limit in the sense of distributions and denote
$$
\rho_0^\theta= \lim_{\e \to 0} \rho_{0,\e}^\theta.  
$$
Then there exist $\delta_1>0$ and $T>0$ such that:
\begin{itemize}
\item for all $\e \in (0,1)$, there is a unique solution $(\rho_\e^\theta, v_\e^\theta)_{\theta \in M}$ of \eqref{fluid} with initial data $(\rho_{0,\e}^\theta, v_{0,\e}^\theta)_{\theta \in M}$, such that $\rho_\e^\theta, v_\e^\theta \in C([0,T]; B_{\delta_1})$ for all $\theta \in M$ and $\e \in (0,1)$, with bounds that are uniform in $\e$;
\item there is a unique solution $(\rho^\theta, v^\theta)_{\theta \in M}$ of \eqref{limit-fluid} with initial data $(\rho_{0}^\theta, v_{0}^\theta)_{\theta \in M}$, such that $\rho^\theta, v^\theta \in C([0,T]; B_{\delta_1})$
 for all $\theta \in M$;
\item for all $s \in \mathbb{N}$, we have
\be
\label{eq:conv}
\sup_{\theta \in M} \sup_{t \in [0,T]} \left[ \| \rho_\e^\theta - \rho^\theta\|_{H^s (\bt)} +  \| v_\e^\theta-\frac{1}{i}(d_+(t,x)e^{\frac{it}{\sqrt \e}}-d_-(t,x)e^{-\frac{it}{\sqrt \e}}) - v^\theta\|_{H^s (\bt)} \right] \to_{\e \to 0 } 0
\ee
where $d_{\pm}(t,x)$ are the correctors introduced to avoid the so called ``plasma oscillations''. They are defined as the solution of
\begin{equation}
\label{eq:correctors}
\curl \ d_{\pm}=0, \qquad \div \bigg(\pt_t d_{\pm}+\left(\int \rho_\theta v_\theta \mu (d\theta) \cdot \na \right)d_{\pm} \bigg) =0,
\end{equation}
\begin{equation}
\label{eq:correctors_initial}
\div d_{\pm} (0)= \underset{\e \to 0}{\lim} \div \frac{\sqrt \e E^\e(0) \pm i j^\e(0)}{2},
\end{equation}
where $ j^\e:= \int \rho^\e_\theta v^\e_ \theta \mu (d\theta).$
\end{itemize}
\end{prop}

\begin{remark} If in \eqref{eq:correctors_initial}, $\div d_{\pm} (0)=0$, then the initial data are said to be well-prepared and there are no plasma oscillations in the limit $\e\to 0$.
\end{remark}


\bigskip

The main result of this paper is the following:

\begin{thm}
\label{thm1}
Let $\gamma$, $\delta_0$, 
and $C_0$ be positive constants.
Consider a sequence $(f_{0,\e})$ of non-negative initial data in $L^1$  for \eqref{vp} such that for all $\e \in (0,1)$, and all $x \in \bt^d$, 
\begin{itemize}
\item (uniform estimates) 
$$
\| f_{0,\e}\|_\infty\leq C_0, \quad \mathcal{E}(f_{0,\e}) \leq C_0,
$$
\item (compact support in velocity)
$$
f_{0,\e}(x,v)= 0 \quad \text{if  } |v| > \frac{1}{\e^\gamma},
$$
\item (analytic + perturbation)
There exists a function $\varphi: (0,1] \to \br^+$, with $\lim_{\e \to 0} \varphi(\e) = 0$ such that the following hold. Assume the following decomposition:
$$
f_{0,\e} = g_{0,\e} + h_{0,\e},
$$
where $(g_{0,\e})$ is a sequence of continuous functions
satisfying
$$
\sup_{\e\in (0,1)}\sup_{v \in \br^d} \, (1+|v|^2) \| g_{0,\e} (\cdot,v)\|_{B_{\delta_0}}  \leq C,
$$
admitting a limit $g_0$ in the sense of distributions.
Furthemore, $(h_{0,\e})$ is a sequence of functions 
satisfying for all $\e>0$
$$
W_2(f_{0,\e},g_{0,\e}) = \varphi(\e).$$

\end{itemize}


For all $\e \in (0,1)$, consider $f_\e(t)$  a global weak solution of \eqref{vp} with initial condition $f_{0,\e}$, in the sense of Arsenev \cite{Ar}. Define the filtered distribution function
\begin{equation} 
\widetilde{f}_\e(t,x,v) := f_\e \Big(t,x,v-\frac{1}{i}(d_+(t,x)e^{\frac{it}{\sqrt \e}}-d_-(t,x)e^{-\frac{it}{\sqrt \e}})\Big)
\end{equation}
where $(d_\pm)$ are defined in \eqref{eq:correctors}.

There exist $T>0$ and $g(t)$ a weak solution on $[0,T]$ of \eqref{formal} with initial condition $g_0$ such that
$$
\lim_{\e \to 0} \sup_{t \in [0,T]}  W_1(\widetilde{f}_\e(t), g(t)) = 0.
$$
Explicitly, we can take
\begin{itemize}
\item in two dimensions, $\varphi(\e)= \exp\left[ \exp\left( - \frac{K}{\e^{2(1+ \max(\beta,\gamma))}}\right) \right]$,
for some constant $K>0$, \ $\beta>2$;

\item in three dimensions, $\varphi(\e)= \exp\left[ \exp\left( - \frac{K}{\e^{2+ \max(38,3\gamma))}}\right) \right]$,
for some constant $K>0$.
\end{itemize}


\end{thm}

\begin{remark} Let us notice that in Theorem \ref{thm1} we consider sequences of initial conditions with compact support in velocity (yet, we allow the support to grow polynomially as $\e \to 0).$  The reason is that, in the spirit of \cite{HI}, we rely on a Wasserstein stability estimate to control the difference between the unperturbed analytic solution and the perturbed one. In dimensions $2$ and $3$, as we shall explain below, 
we need $L^\infty$ bounds on the densities of  \emph{both} solutions. In order to have such a bound on the $L^\infty$ norm of the densities  we need to control the support in velocity. Such a condition was not required in our previous paper \cite{HI} since, in the 1D case, we could use a ``weak-strong'' Wasserstein stability estimate and only a $L^\infty$ bound on the unperturbed solution was needed.
\end{remark}

\begin{remark} In the opposite direction, we recall that in the one dimensional case there is a negative result stating that an initial rate of convergence of the form $\varphi(\e)= \e^s$ for any $s>0$ is not sufficient to ensure the convergence for positive times. This is the consequence of \emph{instability} mechanisms described in \cite{Gr99} and \cite{HKH}. As a matter of fact, we expect that an analogue of this result holds also in higher dimension.
\end{remark}

\section{Overview of the paper}

The following of the paper is entirely devoted to the proof of Theorem \ref{thm1}. 
Let us describe the main steps that are needed to achieve this convergence result.
\begin{enumerate}

\item We first revisit Loeper's Wasserstein stability estimates \cite{Loe} on the torus $\bt^d$ and with quasineutral scaling, which allows us to  control
$W_2 (f_1,f_2)$, where $f_1$ and $f_2$ are two given solutions of \eqref{vp} in terms of the initial distance  $W_2 (f_1(0),f_2(0))$ and of the $L^\infty
$ norm of the densities $\rho_1= \int_{\br^d} f_1 \,dv$ and $\rho_2= \int_{\br^d} f_2 \,dv$. This first step is performed in Section~ \ref{sec:wasserstein}.

\item  In the one dimensional case studied in our previous work \cite{HI} we had a ``weak-strong'' type stability estimate; as a consequence a control of the $L^\infty$ norm of the density of the perturbed solution $f_\e$ (following the notations of Theorem~\ref{thm1}) was not required.

In the higher dimensional case under study, such an estimate is needed. To achieve this, we give quantitative estimates of the growth of the support in velocity for solutions of~\eqref{vp}. We separate the $2$ and the $3$-dimensional case since different tools are involved.

While in the two dimensional case studied in Section~\ref{sec:2D} only elementary considerations are needed, in the three dimensional case, we shall use a more involved bootstrap argument due to Batt and Rein \cite{BR}, see Section~\ref{sec:3D}.

\item We finally conclude in Section~\ref{sec:thm} by combining the results of the two previous steps and Grenier's convergence result stated in Proposition~\ref{grenier}.

\end{enumerate}

\section{Proofs of Steps 1 and 2}

\subsection{$W_2$ stability estimate}
\label{sec:wasserstein}

We start by giving the relevant $W_2$ stability estimate, adapting from the work of Loeper \cite{Loe}.

\begin{thm}
\label{thm:Loeper}
Let $f_1, f_2$ be two weak solutions of the Vlasov-Poisson system \eqref{vp}, and set 
$$
\rho_1:= \int_{\br^d} f_1 \, dv, \quad \rho_2= \int_{\br^d} f_2 \, dv.
$$ 
Define the function
\be
\label{eq:At}
A(t):=\biggl[1+\frac{1}{\e^2}\sqrt{\|\rho_2(t)\|_{L^\infty(\mathbb T^d)} }
\Bigl[\max\bigl\{\|\rho_1(t)\|_{L^\infty(\mathbb T^d)},
 \|\rho_2(t)\|_{L^\infty(\mathbb T^d)}\bigr\} \Bigr]^{1/2}\biggr]+\frac{\|\rho_1(t)-1\|_{L^\infty(\mathbb T^d)}}{\e^2},
 \ee
 and assume that $A(t) \in L^1([0,T])$ for some $T>0$. 
 Also, set
 \be
 \label{eq:Ft}
F_t[z]:= 16\,d\,e^{\log\left(\frac{z}{16\,d}\right)\exp\left[C_0\int_0^t A(s)\,ds\right]}\qquad \forall\,z \in [0,d],\,t \in [0,T].
 \ee
 Then there exists a dimensional constant $C_0 >1$ such that, if $W_2(f_1(0), f_2(0) )\leq d$, then
 for all $t \in [0,T]$,
\be
\label{wass}
W_2(f_1(t), f_2(t) )\leq \left\{
\begin{array}{ll}
F_t[W_2(f_1(0), f_2(0) )] 
&\text{ if }F_T[W_2(f_1(0), f_2(0) )]\leq d,\\
d\,  e^{C_0\int_{0}^t A(s)\,ds}
&\text{ if }F_T[W_2(f_1(0), f_2(0) )]> d.
\end{array}
\right.
\ee
\end{thm}

%
%
%
%
%
%
%

\begin{proof}[Proof of Theorem \ref{thm:Loeper}]

Before starting the proof we recall two important estimates that follow immediately from 
\cite[Theorem 2.7]{Loe} and the analogue of \cite[Lemma 3.1]{Loe} on the torus
(notice that $|x-y| \leq \sqrt{d}$ for all $x,y \in \mathbb T^d$): 
\begin{lem}
\label{lem:Loeper}
Let $\Psi_i:\mathbb T^d\to \mathbb R$ solve
$$
-\e^2 \Delta \Psi_i=\rho_i-1,\qquad i=1,2.
$$
Then
$$
 \e^2 \|\nabla \Psi_1-\nabla \Psi_2\|_{L^2(\mathbb T^d)} \leq \Bigl[\max\bigl\{\|\rho_1\|_{L^\infty(\mathbb T^d)},
 \|\rho_2\|_{L^\infty(\mathbb T^d)}\bigr\} \Bigr]^{1/2}\,W_2(\rho_1,\rho_2),
$$
$$
\e^2|\nabla \Psi_i(x) - \nabla \Psi_i(y)| \leq C\,|x-y|\,\log\biggl(\frac{4 \sqrt{d}}{|x-y|} \biggr) \,\|\rho_i-1\|_{L^\infty(\mathbb T^d)}\qquad\forall\,x,y \in\mathbb T^d,\, i=1,2.
$$
\end{lem}

To prove Theorem \ref{thm:Loeper}, we define the quantity
$$
Q(t):=\int_0^1|Y_1(t,S_1(s))-Y_2(t,S_2(s))|^2\,ds
$$
where
$$
S_1,S_2:[0,1]\to \mathbb T^d\times \mathbb R^d
$$
are measurable maps such that $(S_i)_\# ds=f_i(0)$ and 
$$
W_2(f_1(0),f_2(0))^2=\int_0^1|S_1(s)-S_2(s)|^2\,ds,
$$
while $Y_i=(X_i,V_i)$ solve the ODE
$$
\begin{array}{l}
\dot X_i=V_i,\\
\dot V_i=-\nabla\Psi_i(t,X_i)
\end{array}
$$
with the initial condition $Y_i(0,x,v)=(x,v)$.

Thus, thanks to \cite[Corollary 3.3]{Loe} it follows that $f_i(t)=Y_i(t)_\#f_i(0)=[Y_i(t,S_i)]_\#ds$.

Then we compute
\begin{align*}
\frac{1}{2}\frac{d}{dt} Q(t)&=\int_0^1 [X_1(t,S_1)-X_2(t,S_2)]\,[V_1(t,S_1)-V_2(t,S_2)]\,ds\\
&+\int_0^1 [V_1(t,S_1)-V_2(t,S_2)]\,\bigl[\nabla\Psi_1\bigl(t,X_1(t, S_1)\bigr)-\nabla\Psi_2\bigl(t,X_2(t, S_2)\bigr)\bigr]\,ds.
\end{align*}
By Cauchy-Schwarz inequality, we have
\begin{align*}
\frac{1}{2}\frac{d}{dt} Q(t) &\leq 
\sqrt{\int_0^1|X_1(t,S_1)-X_2(t,S_2)|^2\,ds}\sqrt{\int_0^1|V_1(t,S_1)-V_2(t,S_2)|^2\,ds}\\
&+\sqrt{\int_0^1|V_1(t,S_1)-V_2(t,S_2)|^2\,ds}\sqrt{\int_0^1\bigl|\nabla\Psi_1\bigl(t,X_1(t, S_1)\bigr)-\nabla\Psi_2\bigl(t,X_2(t, S_2)\bigr)\bigr|^2\,ds}\\
&\leq \,Q(t)+\sqrt{Q(t)}\sqrt{\int_0^1\bigl|\nabla\Psi_1\bigl(t,X_1(t, S_1)\bigr)-\nabla\Psi_1\bigl(t,X_2(t, S_2)\bigr)\bigr|^2\,ds}\\
&+\sqrt{Q(t)}\sqrt{\int_0^1\bigl|\nabla\Psi_1\bigl(t,X_2(t, S_2)\bigr)-\nabla\Psi_2\bigl(t,X_2(t, S_2)\bigr)\bigr|^2\,ds}.
\end{align*}
Using the definition of the push-forward, we finally get
\begin{align*}
\frac{1}{2}\frac{d}{dt} Q(t)
&\leq \,Q(t)+\sqrt{Q(t)}\sqrt{\int_0^1\bigl|\nabla\Psi_1\bigl(t,X_1(t, S_1)\bigr)-\nabla\Psi_1\bigl(t,X_2(t, S_2)\bigr)\bigr|^2\,ds}\\
&+\sqrt{Q(t)}\sqrt{\int_{\mathbb T^d\times \mathbb R^d} |\nabla \Psi_1(t,x)-\nabla\Psi_2(t,x)|^2f_2(t,x,v)\,dx\,dv}\\
&\leq \,Q(t)+\sqrt{Q(t)}\sqrt{\int_0^1\bigl|\nabla\Psi_1\bigl(t,X_1(t, S_1)\bigr)-\nabla\Psi_1\bigl(t,X_2(t, S_2)\bigr)\bigr|^2\,ds}\\
&+\sqrt{\|\rho_2(t)\|_{L^\infty(\mathbb T^d)} }\sqrt{Q(t)}\sqrt{\int_{\mathbb T^d} |\nabla \Psi_1(t,x)-\nabla\Psi_2(t,x)|^2\,dx}.
\end{align*}
We now apply Lemma \ref{lem:Loeper} to the last two terms and we bound them respectively by
$$
C\,\frac{\|\rho_1(t)-1\|_{L^\infty(\mathbb T^d)}}{\e^2} \sqrt{Q(t)}\sqrt{\int_0^1\bigl|X_1(t, S_1)-X_2(t, S_2)\bigr|^2\log^2\biggl(\frac{4 \sqrt{d}}{\bigl|X_1(t, S_1)-X_2(t, S_2)\bigr|}\biggr)\,ds}
$$
and
$$
\frac{1}{\e^2}\sqrt{\|\rho_2(t)\|_{L^\infty(\mathbb T^d)} }
\Bigl[\max\bigl\{\|\rho_1(t)\|_{L^\infty(\mathbb T^d)},
 \|\rho_2(t)\|_{L^\infty(\mathbb T^d)}\bigr\} \Bigr]^{1/2}\,W_2(\rho_1(t),\rho_2(t)).
$$
Since $W_2(\rho_1(t),\rho_2(t)) \leq \sqrt{Q(t)}$ (see for instance \cite[Lemma 3.6]{Loe}) we conclude that
\begin{multline*}
\frac{1}{2}\frac{d}{dt} Q(t) \leq \biggl[1+\frac{1}{\e^2}\sqrt{\|\rho_2(t)\|_{L^\infty(\mathbb T^d)} }
\Bigl[\max\bigl\{\|\rho_1(t)\|_{L^\infty(\mathbb T^d)},
 \|\rho_2(t)\|_{L^\infty(\mathbb T^d)}\bigr\} \Bigr]^{1/2}\biggr]\,Q(t)\\
 +C\,\frac{\|\rho_1(t)-1\|_{L^\infty(\mathbb T^d)}}{\e^2} \sqrt{Q(t)}\sqrt{\int_0^1\bigl|X_1(t, S_1)-X_2(t, S_2)\bigr|^2\log^2\biggl(\frac{4 \sqrt{d}}{\bigl|X_1(t, S_1)-X_2(t, S_2)\bigr|}\biggr)\,ds}.
\end{multline*}
Noticing that $\bigl|X_1(t, S_1)-X_2(t, S_2)\bigr| \leq \sqrt{d}$ (since $X_1$ and $X_2$ are points on the torus)
and
$$
\log\biggl(\frac{4 \sqrt{d}}z\biggr) =\frac{1}{2}\log\biggl(\frac{16\,d}{z^2}\biggr) \qquad \,\forall\,z >0,
$$
we get
\begin{multline*}
\int_0^1\bigl|X_1(t, S_1)-X_2(t, S_2)\bigr|^2\log^2\biggl(\frac{4 \sqrt{d}}{\bigl|X_1(t, S_1)-X_2(t, S_2)\bigr|}\biggr)\,ds\\
=\frac{1}{4} \int_0^1\bigl|X_1(t, S_1)-X_2(t, S_2)\bigr|^2\log^2\biggl(\frac{16\,d}{\bigl|X_1(t, S_1)-X_2(t, S_2)\bigr|^2}\biggr)\,ds\\
=\frac{1}{4}\int_0^1g(s)\log^2\biggl(\frac{16\,d}{g(s)}\biggr)\,ds,
\end{multline*}
where we set $g(s):=\bigl|X_1(t, S_1)-X_2(t, S_2)\bigr|^2$.

Hence, since the  function
\begin{equation}
\label{eq:H}
z\mapsto H(z):=\left\{
\begin{array}{ll}
z \log^2\left(\frac{16\,d}{z}\right) & \text{for $0 \leq z \leq d$},\\
d\log^2(16)& \text{for $z \geq d$},\\
\end{array}
\right.
\end{equation}
is concave and increasing,
recalling that $g \leq d$ and 
applying Jensen's inequality to $H$ we get
\begin{align*}
\frac{1}{2}\frac{d}{dt} Q(t)
&\leq \biggl[1+\frac{1}{\e^2}\sqrt{\|\rho_2(t)\|_{L^\infty(\mathbb T^d)} }
\Bigl[\max\bigl\{\|\rho_1(t)\|_{L^\infty(\mathbb T^d)},
 \|\rho_2(t)\|_{L^\infty(\mathbb T^d)}\bigr\} \Bigr]^{1/2}\biggr]\,Q(t)\\
 &\qquad \qquad \qquad \qquad +C\,\frac{\|\rho_1(t)-1\|_{L^\infty(\mathbb T^d)}}{\e^2} \sqrt{Q(t)}\sqrt{H\biggl(\int_0^1g(s)\,ds\biggr)}\\
 &\leq \biggl(1+\frac{1}{\e^2}\sqrt{\|\rho_2(t)\|_{L^\infty(\mathbb T^d)} }
\Bigl[\max\bigl\{\|\rho_1(t)\|_{L^\infty(\mathbb T^d)},
 \|\rho_2(t)\|_{L^\infty(\mathbb T^d)}\bigr\} \Bigr]^{1/2}\biggr)\,Q(t)\\
 &\qquad \qquad \qquad \qquad +C\,\frac{\|\rho_1(t)-1\|_{L^\infty(\mathbb T^d)}}{\e^2} \sqrt{Q(t)}\sqrt{H\bigl(Q(t)\bigr)},
 \end{align*}
where for the last inequality we used that $\int_0^1g(s)\,ds \leq Q(t)$.

In particular, recalling the definition of $A(t)$ in \eqref{eq:At},
by \eqref{eq:H} we deduce that
there exists a dimensional constant $C_0>0$ such that
\be
\label{eq:ODE1}
\frac{d}{dt} Q(t) \leq C_0\,A(t) \,Q(t)\log \biggl(\frac{16\,d}{Q(t)}\biggr)  \qquad \text{as long as $Q(t) \leq d$},
\ee
while
\be
\label{eq:ODE2}
\frac{d}{dt} Q(t) \leq C_0\,A(t) \,Q(t)  \qquad \text{when $Q(t) \geq d$}.
\ee
In particular, assuming $Q(0)\leq d$, by \eqref{eq:ODE1} we get
\be
\label{eq:sol1}
Q(t)\leq 16\,d\,e^{\log\left(\frac{Q(0)}{16\,d}\right)\exp\left[C_0\int_0^t A(s)\,ds\right]}=:F_t[Q(0)]
\ee
as long as $Q(t) \leq d$, which is the case in particular if  $F_t[Q(0)]\leq d$.
On the other hand, if there is some time $t_0$ such that $F_{t_0}[Q(0)]=d$,
since $Q(t_0)\leq F_{t_0}[Q(0)] $ by
 \eqref{eq:ODE2} we get
\be
\label{eq:sol2}
Q(t) \leq d \,e^{C_0\int_{t_0}^t A(s)\,ds} \leq d \,e^{C_0\int_{0}^t A(s)\,ds} \qquad \text{for} \ t\ge t_0.
\ee
Noticing that $F_t$ is monotone in $t$, we deduce in particular that if $F_T[Q(0)]\leq d$ and $Q(0) \leq d$ then \eqref{eq:sol1} holds,
while if $F_T[Q(0)]> d$ and $Q(0) \leq d$ then one can simply apply \eqref{eq:sol2}.
Finally, if $Q(0) > d$ then we apply \eqref{eq:ODE2} to get
$$
Q(t) \leq Q(0) \,e^{C_0\int_0^t A(s)\,ds}.
$$
Combining these three estimates and recalling that $Q(0)=W_2(f_1(0),f_2(0))^2$ while $Q(t) \geq W_2(f_1(t),f_2(t))^2$, this concludes the proof.
\end{proof}

%
%
%

\subsection{Control of the growth of the support in velocity in 2D}
\label{sec:2D}

In this section, $d=2$. Our goal  is to obtain estimates on the growth in time of the support in velocity.
This allows us to get bounds for the $L^\infty$ norms of the local densities on some interval of time $[0,T]$. Recall that in the end, they will be used to apply the Wasserstein stability estimates proved in Section \ref{sec:wasserstein}.

%

For $f_\e$ a solution of \eqref{vp}, define 
$$
V_\e(t):= \sup \left\{ |v|\,: \, v \in \br^2, \, \exists x \in   \bt^2, f_\e(t,x,v) > 0\right\}.
$$
The key point of this Section is the following Proposition.
\begin{prop}
\label{growth-2D}
Suppose that
$$
\|f_\e(0)\|_\infty \leq C_0,\qquad \int\Bigl(|v|^2 + U_\e (0,x)\Bigr)f_\e(0,x,v)\,dv\,dx \leq C_0.
$$
Assume that $V_\e(0)\leq C_0 /\e^{\gamma}$, for some $\gamma>0$.
Let $T>0$ be fixed. For all $\beta >2$, there is $C_\beta>0$, such that we have for all $\e \in (0,1)$ and all $t \in [0,T]$, 
\be
V_\e(t) \leq C_\beta/ \e^{\max\{\beta,\gamma\}}.
\ee
Therefore, for all $\beta >2$, there is $C'_\beta>0$, such that, for all $\e \in (0,1)$ and all $t \in [0,T]$,
\be
\label{rho2D}
\| \rho_\e \|_\infty \leq C'_\beta/ \e^{2 \max\{\beta,\gamma\}}.
\ee
\end{prop}

In order to prove this Proposition, we shall use for convenience the change of variables $(t,x,v) \mapsto (\frac{t}{\e}, \frac{x}{\e}, v)$. This leads us to consider, the following Vlasov-Poisson system, for $(x,v) \in \frac{1}{\e}\bt^2 \times \br^2$:
\be
\label{vp'}
 \left\{ \begin{array}{cccc}\pt_t g_\e+v\cdot \na_x g_\e+ F_\e\cdot \na_v g_\e=0,  \\
F_\e=-\na_x \Phi_\e, \\
-\Delta_x \Phi_\e=\int_{\br^2} g_\e\, dv - \int_{\frac{1}{\e} \bt^2 \times \br^2}  g_\e\, dv \, dx ,\\
g_\e\vert_{t=0}=g_{0,\e}\ge0,\\
  \int_{\frac{1}{\e} \bt^2 \times \br^2} g_{0,\e}\,dx\,dv=\frac{1}{\e^2}.
\end{array} \right.
\ee
We shall denote
$$
\eta_\e := \int_{\br^2} g_\e \, dv
$$
and define for all $t\ge 0$,
\be
V(t) := \sup \left\{ |v|, \, v \in \br^2, \, \exists x \in  \frac{1}{\e} \bt^2, g_\e(t,x,v) > 0\right\}=V_\e\biggl(\frac{t}{\e}\biggr).
\ee
In the following, for brevity, the notation $\L^p$, for $p \in [1,+\infty]$, will stand for $L^p\left(\frac{1}{\e} \bt^2 \times \br^2\right)$ or $L^p\left(\frac{1}{\e} \bt^2\right)$, depending on the context.

The main goal is now to prove the following Proposition, from which it is straightforward to deduce Proposition \ref{growth-2D}  by applying the result for $t= \frac{T}{\e}$.
\begin{prop}
\label{prop-growth2D}
Let $\beta>0$. Let $C_0>0$ such that for all $\e \in (0,1)$,
\be
\| g_{\e}(0)\|_{\L^\infty} \leq C_0, \quad \int_{\frac{1}{\e} \bt^2 \times \br^2}  \Bigl(|v|^2 + \Phi_\e (0,x)\Bigr)g_{\e}(0) \, dv dx \leq \frac{C_0}{\e^2}.
\ee
Then for all $\a \in (0,1)$, there exist $C_{\a},C_\a'>0$ such that for all $\e \in (0,1)$, and all $t \geq 0$, 
\be
V(t) \leq  \left(\frac{C_\a}{\e^{\a+1}}t + (1+ V(0))^{1-\a}\right)^{1/(1-\a)} -1.
\ee
\end{prop}

We will use the following standard property of conservation of $L^p$ norms and energy for solutions to the Vlasov-Poisson system in the sense of Arsenev:
\begin{lem}For all $t\geq 0$, we have
\be
\label{apriori}
\| g_{\e}(t)\|_{\L^\infty} \leq C_0, \quad \int_{\frac{1}{\e} \bt^2 \times \br^2}  \Bigl(|v|^2 + \Phi_\e (t,x)\Bigr)g_{\e}(t) \, dv dx \leq \frac{C_0}{\e^2}.
\ee
\end{lem}

We also rely on the following Lemma about the Green kernel of the Laplacian on $\frac{1}{\e} \bt^2$ (obtained from standard results on the Green kernel of the Laplacian on $\bt^2$, after rescaling). We refer for instance to Caglioti and Marchioro \cite{CaMa}.
\begin{lem}
\label{kernel}
There exists $K_0 \in C^\infty(\bt^2;\br^2)$ such that, denoting
$$
K_\e(x) = \frac{1}{2\pi} \frac{x}{|x|^2} + \e K_0 (\e x),
$$
we have for all $x \in [-1/\e,1/\e]^2$,
$$
F_\e(x) = \int_{[-1/\e,1/\e]^2} K_\e(x-y)  \big[\eta_\e(y)-1\big] \, dy.
$$
\end{lem}

The key ingredient is the following Lemma, in which we obtain some appropriate $\L^\infty$ bound for the electric field, which allows us to control the growth of the support in velocity.
\begin{lem}
\label{lem-firstbounds}
We have the following bounds.
\begin{enumerate}
\item There is a constant $C_1>0$ such that for all $\e \in (0,1)$, and all $t \geq 0$,
\be
\label{bound:eta}
\|\eta_\e(t)\|_{\L^2} \leq \frac{C_1}{\e}, \quad \|\eta_\e(t)\|_{\L^\infty} \leq {C_1} V(t)^2.
\ee
\item There is a constant $C_2>0$ such that for all $\e \in (0,1)$, and all $t\geq 0$,
\be
\| F_\e(t,\cdot) \|_{\L^\infty} \leq C_2 \left(1+  \frac{1}{\e} \left[\log \frac{1}{\e}(1+V(t))\right]^{1/2} \right).
\ee

\item For any $\a>0$, there is a constant $C_\a$ such that for all $0\leq t' \leq t $,
\be
\label{eq:Ca}
V(t) \leq V(t') + \frac{C_\a}{\e^{1+\a}}\int_{t'}^{t} (1+ V(s))^\a \, ds.
\ee
\end{enumerate}
\end{lem}

\begin{proof}[Proof of Lemma \ref{lem-firstbounds}]In this proof, $C>0$ will stand for an universal constant that may change from line to line.
\begin{enumerate}
\item By the following interpolation argument, we have for all $R>0$
$$
\eta_\e=\int_{\br^2}g_\e\,dv = \int_{|v|\leq R}g_\e\,dv+ \int_{|v|> R}g_\e\,dv
\leq \|g_\e\|_\infty R^2 + \frac{1}{R^2}\int_{\br^2}|v|^2g_\e\,dv,
$$
so by optimizing with respect to $R$ we deduce that
 there is a $C>0$ such that for all $t \geq 0, x \in \frac{1}{\e}\bt^2$,
$$
|\eta_\e|(t,x) \leq C \left( \int g_\e(t,x,v) |v|^2 \, dv \right)^{1/2}
$$
By \eqref{apriori}, we deduce the first estimate of \eqref{bound:eta}.

For what concerns the $\L^\infty$ estimate for $\eta_\e$, it is a plain consequence of the definition of $V(t)$, which controls the support in velocity.

\item By Lemma \ref{kernel}, there holds for all $t \geq 0, x \in [-1/\e,1/\e]^2$,
$$
|F_\e|(t,x) \leq \frac{1}{2\pi} \int_{[-1/\e,1/\e]^2} \frac{1}{|x-x'|} |\eta_\e(t,x')-1|\, dx' + \e \| K_0\|_\infty \|(\eta_\e-1)\|_{\L^1}.\\
$$
We observe that, since $\|(\eta_\e-1)\|_{\L^1(\frac{1}{\e}\bt^2)} \leq \|\eta_\e\|_{\L^1(\frac{1}{\e}\bt^2)}+\|1\|_{\L^1(\frac{1}{\e}\bt^2)}=\frac{2}{\e^2}$,
$$
\e \| K_0\|_\infty \|(\eta_\e-1)\|_{\L^1} \leq \frac{C}{\e}.
$$
Let $R>0$ to be fixed later. We have, using the Cauchy-Schwarz inequality,
\begin{align*}
 \int_{[-1/\e,1/\e]^2} &\frac{1}{|x-x'|} |\eta_\e(t,x') -1| \, dx' \\
 &=  \int_{|x-x'|<R} \frac{1}{|x-x'|} |\eta_\e(t,x')-1| \, dx' +  \int_{|x-x'|\geq R} \frac{1}{|x-x'|} |\eta_\e(t,x')-1| \, dx' \\
 &\leq C R (\|\eta_\e\|_{\L^{\infty}}+1) + \left(\|\eta_\e\|_{\L^{2}} + \frac{1}{\e}\right) \left(\int_{|x'|\geq R, \, x' \in [-1/\e,1/\e]^2} \frac{1}{|x'|^2}  \, dx'\right)^{1/2} \\
 &\leq C R (\|\eta_\e\|_{\L^{\infty}}+1) + \left(\|\eta_\e\|_{\L^{2}} + \frac{1}{\e}\right) \left(\log \frac{C}{\e R} \right)^{1/2}.
\end{align*}
We choose $R= \frac{1}{\|\eta_\e\|_{\L^{\infty}}+1}$, which yields, 
$$
\|F_\e\|_{\L^\infty} \leq C \left(1 + \left(\|\eta_\e\|_{\L^{2}} + \frac{1}{\e}\right) \left(\log \frac{1}{\e} (1+ \|\eta\|_{\L^{\infty}}) \right)^{1/2}\right).
$$
Using Point 1., we obtain the claimed estimate.
\item Let $\a>0$. Let $0\leq t' \leq t $ and let $(x,v)$ such that $g_\e(t',x,v) \neq 0$. Introduce the characteristics $(X(s,t',x,v),\xi(s,s,x,v))$ satisfying for $s \geq t'$ for the system of ODEs
\be
\label{charac}
 \left\{ 
\begin{aligned}
&\frac{d}{ds} X(s,t',x,v)= \xi(s,t',x,v), \quad X(t',t',x,v)=x,\\
&\frac{d}{ds} \xi(s,t',x,v)= F_\e(s, X(s,t',x,v)), \quad \xi(t',t',x,v)=v.
\end{aligned}
\right.
\ee
We have
$$
\xi(t,t',x,v) = v + \int_{t'}^{t}F_\e(s,X(s,t',x,v)) \, ds.
$$
Therefore, we have, using Point 2.,
\begin{align*}
|\xi|(t,t',x,v) &\leq  |v| + \int_{t'}^{t}|F_\e|(s,X(s,t',x,v)) \, ds \\
&\leq |v| + \frac{C_2}{\e}\int_{t'}^{t} \left(1+   \left[\log \frac{1}{\e}(1+V(s))\right]^{1/2} \right) \, ds.
\end{align*}
Since $g_\e$ satisfies \eqref{vp'}, it is thus constant along the characteristics \eqref{charac}, and we have
$$
g_\e(t,X(t,t',x,v),\xi(t,t',x,v))= g_\e(t', x,v).
$$
By definition of $V(t)$ and $V(t')$, we obtain
$$
V(t) \leq V(t') + \frac{C_2}{\e}\int_{t'}^{t}\left(1+   \left[\log \frac{1}{\e}(1+V(s))\right]^{1/2} \right) \, ds.
$$
In order to get the polynomial bound on $V(t)$, let $r_\a>0$ such that for all $x \geq 1$,
$$
(\log x )^{1/2} \leq 1+ r_\a x^\a.
$$
We thus get
$$
V(t) \leq V(t') + \frac{C_2(2+ r_\a)}{\e^{1+\a}}\int_{t'}^{t}  (1+V(s))^{\a}  \, ds,
$$
which proves our claim, taking $C_\a:= C_2(2+ r_\a)$.
 \end{enumerate}
\end{proof}

Equipped with this result, we can finally proceed with the proof of Proposition \ref{prop-growth2D}.
\begin{proof}[Proof of Proposition \ref{prop-growth2D}]

We begin by observing that dividing by $1/(t-t')$ in both sides of \eqref{eq:Ca} and  letting $t'\to t$ 
we deduce that
\be
\label{eq:derV}
\frac{d}{dt} V(t) \leq \frac{C_\a}{\e^{1+\a}}(1+V(t))^\a.
\ee
We will obtain the claimed bound by a comparison principle. To this end, introduce a small parameter $\mu>0$ and define
$$
W_\mu(t):= \left( \frac{2[C_\a+\mu](1-\a)}{\e^{\a+1}} t + (1+V(0)+ \mu)^{1-\a}\right)^{\frac{1}{1-\a}}.
$$
By construction, it satisfies for $t \geq 0 $
\be
\label{eq:derW}
\frac{d}{dt}  W_\mu(t)= \frac{C_\a+\mu}{\e^{1+\a}}  W_\mu(t)^\a.
\ee
and 
$$
W_\mu(0) =1 + V(0) + \mu >1+ V(0).
$$
We claim that $W_\mu(t)\geq 1+V(t)$ for all $t$.
Indeed, let
$$
t_0:=\inf \{t>0\,:\,W_\mu(t)< 1+V(t)\},
$$ 
and assume by contradiction that $t_0<+\infty$.
Notice that because $\mu>0$ we have $t_0>0$.
Then by continuity at the time $t_0$ we get
$$
W_\mu(t_0)= 1+V(t_0).
$$
By \eqref{eq:derV} and \eqref{eq:derW}, we have
\begin{align*}
\frac{d}{dt} (V (t) - W_\mu(t) -1)_{|t=t_0} &\leq \frac{C_\a}{\e^{1+\a}}(1+V(t_0))^\a - \frac{C_\a+\mu}{\e^{1+\a}}  W_\mu(t_0)^\a \\
&\leq - \frac{\mu}{\e^{1+\a}}  W_\mu(t_0)^\a  < 0.
\end{align*}
This is a contradiction with the definition of $t_0$.

Hence we obtained that for all $\mu>0$ and all $t\geq 0$,
$$
1+ V(t) \leq W_\mu(t).
$$
By taking the limit $\mu\to 0$, one finally gets for all all $t\geq 0$
$$
1+ V(t) \leq  \left( \frac{2C_\a(1-\a)}{\e^{\a+1}} t + (1+V(0))^{1-\a}\right)^{\frac{1}{1-\a}},
$$
which proves the Proposition.
\end{proof}


%
%
%
%
%
%

\subsection{Control of the growth of the support in velocity in 3D using Batt and Rein's estimates}
\label{sec:3D}

In this section, we deal with the case $d=3$. We consider as before, for $f_\e$ a solution of \eqref{vp},  
$$
V_\e(t):= \sup \left\{ |v|, \, v \in \br^3, \, \exists x \in   \bt^3, f_\e(t,x,v) > 0\right\}.
$$
We have in 3D the analogue of the key Proposition \ref{growth-2D} in 2D. 
\begin{prop}
\label{3Dcrucial}Suppose that
$$
\|f_\e(0)\|_\infty \leq C_0,\qquad \int\Bigl(|v|^2 + U_\e (0,x)\Bigr)f_\e(0,x,v)\,dv\,dx \leq C_0.
$$
Assume that $V_\e(0)\leq C_0 /\e^{\gamma}$, for some $\gamma>0$.
Let $T>0$ be fixed. There is $C'>0$, such that we have for all $\e \in (0,1)$ and all $t \in [0,T]$, 
\be
\label{3Dgrowth}
V_\e(t) \leq \frac{C'}{\e^{\max\{38/3,\gamma\}}}.
\ee
Then
\be
\label{rho3D}
\| \rho_\e \|_\infty \leq \frac{C'}{\e^{\max\{38, 3\gamma\}}}.
\ee
\end{prop}
Note that this result involves exponents which are ``more degenerate'' than in the 2-D case.
We shall prove this result as an application of the estimates obtained by Batt and Rein in \cite{BR}.
The result of \cite{BR} is an adaptation to the case of the torus $\bt^3$ of the fundamental contribution of Pfaffelmoser \cite{Pfa} (see also \cite{Sch,Hor}), which allowed to build global classical solutions of the Vlasov-Poisson system in $\br^3 \times \br^3$. In $\br^3 \times \br^3$, it may be possible to get better estimates than \eqref{3Dgrowth} (i.e. with smaller exponents) by using dispersive effects, see \cite{Pfa,Sch,Hor} and more recently \cite{Pal}.
\medskip

In order to prove this Proposition, we shall use the change of variables $(t,x,v) \mapsto (\frac{t}{\e}, \frac{x}{\e}, v)$. This leads us to consider, the following Vlasov-Poisson system, for $(x,v) \in \frac{1}{\e}\bt^3 \times \br^3$:
\be
\label{vp"}
 \left\{ \begin{array}{ccc}\pt_t g_\e+v\cdot \na_x g_\e+ F_\e\cdot \na_v g_\e=0,  \\
F_\e=-\na_x \Phi_\e, \\
-\Delta_x \Phi_\e=\int_{\br^3} g_\e\, dv - \int_{\frac{1}{\e} \bt^3 \times \br^3}  g_\e\, dv \, dx ,\\
g_\e\vert_{t=0}=g_{0,\e}\ge0,\ \  \int_{\frac{1}{\e} \bt^3 \times \br^3} g_{0,\e}\,dx\,dv=\frac{1}{\e^3}.
\end{array} \right.
\ee
We shall denote as in the 2D case
$$
\eta_\e := \int_{\br^3} g_\e \, dv,
$$
and define for all $t\ge 0$,
\be
V(t) := \sup \left\{ |v|, \, v \in \br^3, \, \exists x \in  \frac{1}{\e} \bt^3, g_\e(t,x,v) > 0\right\}.
\ee
As before, in what follows we use the notation $\L^p$, for $p \in [1,+\infty]$, will stand for $L^p\left(\frac{1}{\e} \bt^3 \times \br^3\right)$ or $L^p\left(\frac{1}{\e} \bt^3 \right)$, depending on the context.

The main goal is now to prove the following Proposition, from which we deduce Proposition \ref{3Dcrucial} by choosing $t= \frac{T}{\e}$.
\begin{prop}
\label{prop-growth3D}
Suppose that
$$
\|f_\e(0)\|_{\L^\infty}  \leq C_0,\qquad \int_{\frac{1}{\e} \bt^3 \times \br^3} \Bigl(|v|^2 + \Phi_\e (0,x)\Bigr)f_\e(0,x,v)\,dv\,dx \leq C_0.
$$
Let $\gamma>0$. Let $C_0>0$ such that for all $\e \in (0,1)$,
\be
V(0) \le \frac{C_0}{\e^\gamma}.
\ee
Then there exists $C_1>0$ that for all $\e \in (0,1)$, and all $t \in [0,T]$, 
\be
V(t) \leq  \max\Big\{ \frac{C_0}{\e^\gamma} +   \left[- \frac{C_1}{\e^{32/3}} + \sqrt{\frac{C_1^2}{\e^{64/3}} T^4 +  4\frac{C_1}{\e^{32/3+ \gamma}} }\right], \,  \frac{C_0}{\e^\gamma}  + T^{-7/2} \Big\}.
\ee
\end{prop}



\begin{proof}[Proof of Proposition \ref{prop-growth3D}]
Consider the usual notations for characteristics of \eqref{charac} and introduce, as in \cite{BR},
\begin{align*}
&h_1(t):= \sup \{ \| \eta_\e(s) \|_{L^\infty}, \, 0 \leq s \leq t\} +1 , \\
&h_2(t):= \sup \{ |\xi(s,\tau,x,v)- v|, \, 0\leq s, \tau \leq t, \, (x,v) \in \frac{1}{\e} \bt^3 \times \br^3\}.
\end{align*}
We look for a bound on $h_2$, which will imply the control on $V(t)$.
To this end, we crucially rely on the key bootstrap result in the paper of Batt and Rein \cite{BR}, which we recall in the form of a lemma for the reader's convenience.

\begin{lem}[Batt, Rein]
\label{lem-BR}
Assume that there is $C^*>0$ and $\beta >0$ such that
 $$
 h_2(t) \leq C^* t  h_1(t)^\beta
 $$
then for some universal constant $C>0$ (that hereafter may change from line to line),
 \be
 \label{crucial}
 h_2(t) \leq C t \Big( C^{* \, 4/3} h_1^{2\beta/3}(t) + \frac{1}{\e^3}\left( h_1^{1/6}(t) + \frac{1}{C^*} \right) \Big),
\ee
if $h_1(t)^{-\beta/2} \leq t$.
\end{lem}

By using \cite[Eq. (5), Section 4, p.414]{BR}, there is $C>0$ independent of $\e$ such that for all $\e>0$ and $t \geq 0$, 
 \be
 \label{init}
h_2(t) \leq C  t  h_1(t)^{4/9}.
 \ee
 We deduce from Lemma \ref{lem-BR} and \eqref{crucial} that
  $$
h_2(t) \leq \frac{C}{ \e^3} t  h_1(t)^{8/27}, \quad \text{if  } h_1(t)^{-2/9} \leq t.
 $$
 Using \eqref{crucial} twice, we finally obtain
  $$
h_2(t) \leq \frac{C}{ \e^4}  t h_1(t)^{16/81},  \quad \text{if  } h_1(t)^{-4/27} \leq t
$$
 and since $\frac{32}{243}< \frac{1}{6}$, we get
 $$
h_2(t) \leq  \frac{C}{\e^{16/3}}  t h_1(t)^{1/6},  \quad \text{if  } h_1(t)^{-8/81} \leq t,
 $$
 Using the straightforward bound 
 $$
 h_1(t) \leq C \left( V(0) + h_2(t)\right)^3,
 $$
 we deduce that for all $\e \in (0,1)$ and $t > 0$, 
 $$
h_2(t) \leq \frac{C_1}{\e^{16/3}}  t  \left( \frac{1}{\e^\gamma} +h_2(t)\right)^{1/2},
 $$
if $h_1(t)^{-8/81} \leq t$. On the other hand, if $h_1(t)^{-8/81} > t$ then,
$$
h_1(t) < \frac{1}{t^{8/81}},
$$
and thus, by \eqref{init}, we get
$$
h_2(t) \leq C t^{-7/2}.
$$
We conclude that for all $\e \in (0,1)$ and $t > 0$,
 $$
h_2(t) \leq C \max\Big\{ \frac{1}{\e^{16/3}}  t  \left( \frac{1}{\e^\gamma} +h_2(t)\right)^{1/2}, \, t^{-7/2} \Big\},
 $$
 which yields, taking $t=T$,
 $$
h_2(T) \leq  \max\Bigg\{ \frac{1}{2} \left[- \frac{C_1}{\e^{32/3}} + \sqrt{\frac{C_1^2}{\e^{64/3}} T^4 + 4 \frac{C_1}{\e^{32/3+ \gamma}} }\right], \, T^{-7/2} \Bigg\}.
 $$
Therefore, for $t \in [0,T]$ we get
$$
V(t) \leq  \max\Bigg\{ \frac{C_0}{\e^\gamma} +  \left[- \frac{C_1}{\e^{32/3}} + \sqrt{\frac{C_1^2}{\e^{64/3}} T^4 + 4 \frac{C_1}{\e^{32/3+ \gamma}} }\right], \,  \frac{C_0}{\e^\gamma}  +  T^{-7/2} \Bigg\},
$$
which proves Proposition \ref{prop-growth3D}.
\end{proof}

%

\section{Proof of Theorem \ref{thm1}}
\label{sec:thm}

We prove the main Theorem by a perturbation argument, relying on the Wasserstein stability estimates of Theorem~\ref{thm:Loeper}.

Before that, we first state a Lemma about the effect of $x$-dependent translations in the velocity variable on the $W_1$ distance.
\begin{lem}
\label{lem:translation}
 Let $\mu, \nu$ be probability measures on $\mathcal{P}_1(\mathbb T^d \times \br^d)$ and let $\tilde \mu, \tilde \nu$ be probability measures on $\mathcal{P}_1(\mathbb T^d \times \br^d)$ defined as follows:
$$
<\tilde \mu, \phi(x,v)>:=<\mu, \phi(x, v-C(x))> \qquad \text{for all} \ \phi \in \text{Lip}\,(\mathbb T^d \times \br^d);
$$
$$
<\tilde \nu, \phi(x,v)>:=<\nu, \phi(x, v-C(x))>  \qquad \text{for all} \ \phi \in \text{Lip}\,(\mathbb T^d \times \br^d).
$$
Then 
\be\label{eq:invariance}
W_1(\tilde \mu, \tilde \nu) \le \(1+ \|D_{x}C\|_{L^\infty} \) W_1(\mu, \nu) \qquad \text{where} \ D_{x}C:= \big(\pt_{x_i} C_j(x) \big)_{0\le i, j\le d}.
\ee
\end{lem}
\begin{proof}
By the Kantorovich duality we have the following expression:
$$
W_1(\tilde \mu, \tilde \nu)= \underset{\| \vp\|_{Lip}\le 1}{\sup} \bigg[ \int \vp (x, v+C(x)) d\,\mu(x)  -\vp (x, v+C(x)) d\,\nu(x)\bigg].
$$
Let us denote $\psi(x, v)= \vp(x, v+C(x))$ and computing the gradient of $\psi$ we deduce that
$$
\|\psi \|_{Lip} \le \big(1+\|D_{x}C\|_{L^\infty} \big)\|\vp \|_{Lip}\le \big(1+\|D_{x}C\|_{L^\infty} \big)
$$
from which we obtain \eqref{eq:invariance}.
\end{proof}

We can now proceed with the proof of Theorem~\ref{thm1}. Let $f_{0,\e}, g_{0, \e}, h_{0,\e}$ satisfy the hypotheses of Theorem \ref{thm1}. Using the same notations of the statement, we want to show that for some $T>0$,
$$
\lim_{\e \to 0} \sup_{t \in [0,T]}  W_1(\widetilde{f}_\e(t), g(t)) = 0.
$$
In analogy with the definition of $\widetilde f_{\e}$ we define
$$g_\e(t,x,v) = \int_\M \rho_\e^\theta(t,x) \delta_{v= v^\theta_\e(t,x)} \, d\mu(\theta)$$
and
$$\widetilde g_\e(t,x,v) = \int_\M \rho^\theta_\e(t,x) \delta_{v= v^\theta_\e(t,x)+C_\e(t,x)} \, d\mu(\theta)$$
where $C_\e(t,x):= -\frac{1}{i}(d_+(t,x)e^{\frac{it}{\sqrt \e}}-d_-(t,x)e^{-\frac{it}{\sqrt \e}}).$

We now prove the following estimate:
\be
\label{triangle}
W_1(\widetilde f_{\e}, g)\le W_1(\widetilde f_{\e}, \widetilde g_\e)+W_1(\widetilde g_\e,g).
\ee
We first consider the second term in the right hand side. 
The uniform convergence to $0$ follows from Proposition~\ref{grenier}, and the Sobolev embedding theorem. We have indeed for some $T>0$, for all $t\in [0,T]$:
\begin{align*}
&W_1(\widetilde g_\e (t), g(t)) = \sup_{\|\varphi\|_{\text{Lip}} \leq 1} \langle \widetilde g_\e -g, \, \varphi \rangle \\
&=  \sup_{\|\varphi\|_{\text{Lip}} \leq 1} \left\{ \int_{\bt^d }  \int_\M (\rho_\e^\theta(t,x)\varphi(x,v_\e^\theta(t,x)+ C_\e(t,x)) - \rho^\theta(t,x)\varphi(x,v^\theta(t,x))) \, d\mu(\theta)  \,  dx \right\} \\
&=\sup_{\|\varphi\|_{\text{Lip}} \leq 1} \left\{ \int_{\bt^d }  \int_\M \rho_\e^\theta(t,x)(\varphi(x,v_\e^\theta(t,x) + C_\e(t,x))-\varphi(x,v^\theta(t,x))  \, d\mu(\theta)  \,  dx \right\}  \\
&+ \sup_{\|\varphi\|_{\text{Lip}} \leq 1} \left\{ \int_{\bt^d}  \int_\M (\rho_\e^\theta(t,x) - \rho^\theta(t,x))\varphi(x,v^\theta(t,x)) \, d\mu(\theta)  \,  dx \right\}.
\end{align*}
Thus, we deduce the estimate
\begin{align*}
W_1(\widetilde g_\e (t), g(t)) 
&\leq \sup_{\|\varphi\|_{\text{Lip}} \leq 1}  \sup_{\e \in (0,1), \, \theta \in \M} \|\rho_\e^\theta\|_{\infty} \|\varphi\|_{\text{Lip}} \int_\M \| v_\e^\theta(t,x)+ C_\e(t,x)-v^\theta(t,x)\|_\infty  \, d\mu(\theta) \\
&+ \sup_{\|\varphi\|_{\text{Lip}} \leq 1}  \int_\M \|\rho_\e^\theta- \rho_\theta\|_{\infty} \, d\mu(\theta) \|\varphi\|_{\text{Lip}}  \left(1/2+ \sup_{\theta \in \M} \|v^\theta(t,x)\|_\infty\right) \\
&+  \sup_{\|\varphi\|_{\text{Lip}} \leq 1} \left\{ \int_{\bt^d \times \br}  \int_\M (\rho_\e^\theta(t,x) - \rho^\theta(t,x))\varphi(0,0) \, d\mu(\theta)  \,  dx \right\}.
\end{align*}
We notice that the last term is equal to $0$ since for all $t \geq 0$,
$$
\int_{\bt^d} \int_\M \rho_\e^\theta(t,x) \, d\mu(\theta) \, dx =\int_{\bt^d} \int_\M \rho_\e^\theta(t,x) \, d\mu(\theta) \, dx =1,
$$
by conservation of the total mass. Considering the supremum in time, we  see that the other two terms converge to $0$, using  \eqref{eq:conv}, so that we get
$$ 
\lim_{\e \to 0} \sup_{t \in [0,T]} W_1(\widetilde g_\e,g) =0.
$$
We thus focus on the first term of the right hand side of \eqref{triangle}. First we use Lemma \ref{lem:translation} to see that
$$
W_1(\widetilde f_{\e}, \widetilde g_\e) \leq \(1+ \|D_{x}C_\e(t,x) \|_{L^\infty} \) W_1(f_{\e}, g_{\e}).
$$
Observe from the definition of the corrector $C_\e$, there is $C_T>0$ independent from $\e$ such that for all $t \in [0,T]$,
$$
 \|D_{x}C_\e(t,\cdot) \|_{L^\infty} \leq C_T.
$$
We therefore have to study $W_1(f_{\e}, g_{\e})$. We first use the rough bound
$$
W_1(f_{\e}, g_{\e}) \leq W_2(f_{\e}, g_{\e}),
$$
then use Theorem \ref{thm:Loeper} to get the estimate
$$
\sup_{t \in [0,T]} W_2(f_{\e}, g_{\e}) \le 16\,d\, \exp \left\{\log\left(\frac{\varphi(\e)}{16\,d}\right)\exp\left[C_0T \frac{1}{\e^2} \Big(1 + \| \rho_{f_\e} \|_{L^\infty([0,T]; L^\infty_x)} + \| \rho_{g_\e} \|_{L^\infty([0,T]; L^\infty_x)} \Big) \right]\right\},
$$
where 
$$
\rho_{f_\e} := \int_{\br^d} f_\e \, dv, \quad \rho_{g_\e} = \int_{\M} \rho^\theta_\e \, d\mu(\theta).
$$
Recalling \eqref{eq:conv}, we have for some $C>0$ independent of $\e$ that
$$ \| \rho_{g_\e} \|_{L^\infty([0,T]; L^\infty_x)} \leq C.$$
For what concerns $\rho_{f_\e}$ we apply 
\begin{itemize}
\item in two dimensions, \eqref{rho2D} in Proposition \ref{growth-2D} to infer that for all $\beta>2$, there is some $C_\beta>0$ independent of $\e$ such that
$$
\| \rho_\e \|_\infty \leq \frac{C_\beta}{\e^{2 \max\{\beta,\gamma\}}};
$$
\item  in three dimensions, \eqref{rho3D} in Proposition \ref{3Dcrucial} to infer that for some $C>0$ independent of $\e$ such that
$$
\| \rho_\e \|_\infty \leq \frac{C}{\e^{\max\{38, 3\gamma\}}}.
$$
\end{itemize}
We deduce that choosing
\begin{itemize}
\item in two dimensions, $\varphi(\e)= \exp\left[ \exp\left( - \frac{K}{\e^{2(1+ \max(\beta,\gamma))}}\right) \right]$,
for some constant $K>0$;

\item in three dimensions, $\varphi(\e)= \exp\left[ \exp\left( - \frac{K}{\e^{2+ \max(38,3\gamma))}}\right) \right]$,
for some constant $K>0$,
\end{itemize}
up to take a smaller time interval of convergence $[0,T]$,
$$
\lim_{\e \to 0} \sup_{t \in [0,T]} W_2( f_\e,g_\e) =0.
$$
We conclude that 
$$
\lim_{\e \to 0} \sup_{t \in [0,T]} W_1( \tilde{f}_\e,g) =0
$$
and the proof of Theorem~\ref{thm1} is complete.

\bibliographystyle{plain}
\bibliography{strong-strong-MultiD}

\end{document}